\pdfoutput=1
\RequirePackage{ifpdf}
\ifpdf 
\documentclass[pdftex]{sigma}
\else
\documentclass{sigma}
\fi

\usepackage{slashed}

\newcommand{\lie}[1]{\mathfrak{#1}}
\def\End{\operatorname{End}}
\def\coker{\operatorname{coker}}
\newcommand{\R}{\mathbb{R}}
\newcommand{\C}{\mathbb{C}}
\def\Hom{\operatorname{Hom}}
\def\tr{\operatorname{tr}}
\newcommand{\PP}{{\rm P}}
\newcommand{\OO}{{\mathcal O}}

\begin{document}
\allowdisplaybreaks

\renewcommand{\thefootnote}{}

\newcommand{\arXivNumber}{2210.17512}

\renewcommand{\PaperNumber}{003}

\FirstPageHeading

\ShortArticleName{A Note on Coupled Dirac Operators}

\ArticleName{A Note on Coupled Dirac Operators\footnote{This paper is a~contribution to the Special Issue on Differential Geometry Inspired by Mathematical Physics in honor of Jean-Pierre Bourguignon for his 75th birthday. The~full collection is available at \href{https://www.emis.de/journals/SIGMA/Bourguignon.html}{https://www.emis.de/journals/SIGMA/Bourguignon.html}}}

\Author{Nigel J.~HITCHIN}

\AuthorNameForHeading{N.J.~Hitchin}

\Address{Mathematical Institute, Woodstock Road, Oxford OX2 6GG, UK}
\Email{\href{mailto:hitchin@maths.ox.ac.uk}{hitchin@maths.ox.ac.uk}}
\URLaddress{\url{https://people.maths.ox.ac.uk/hitchin/}}

\ArticleDates{Received November 01, 2022, in final form January 10, 2023; Published online January 13, 2023}

\Abstract{The article considers some concrete solutions to the Dirac equation coupled to a vector bundle with connection, arising in the study of Yang--Mills equations and vector bundles on Riemann surfaces.}

\Keywords{Dirac equation; spinor; Yang--Mills; holomorphic structure; index theorem}

\Classification{53C07; 14D20}

\begin{flushright}
\begin{minipage}{65mm}
\it Dedicated to Jean-Pierre Bourguignon\\ on the occasion of his 75th birthday
\end{minipage}
\end{flushright}

\section{Introduction}
I first got to know Jean-Pierre in Paris in the mid 1970s when I visited Berger's group in Riemannian geometry. At the time I was following up my thesis topic, working on the Dirac operator in various contexts. In the years that followed (when Jean-Pierre was not occupied with his tireless work for the international community of mathematicians) he shared my interest and, as evidenced by \cite{Bourguignon,BourguignonJ/GauduchonP,BourguignonJ/HijaziO/MilhoratJ/MoroianuA/MoroianuS} for example, he was equally intrigued by questions involving spinors -- geometrically less accessible than vectors and tensors but often quite powerful in applications.

My early focus was on index theory and Weitzenb\"ock formulas, which provide existence theorems for solutions of the Dirac equation. In this note we consider some concrete examples of the Dirac operator coupled to a vector bundle with connection, and encounter situations where the null space has bigger then the expected value given by the index theorem, cases related to my mathematical interests over the years.

The first concerns the ${\rm SU}(2)$ Yang--Mills equations on~$S^4$. A vanishing theorem for the Dirac equation coupled to the fundamental representation gives one of the vector spaces which plays a fundamental role in the ADHM construction of the self-dual solutions, and in fact coupling to any representation the index theorem determines the dimension of the space of solutions. For many years an outstanding question was whether all Yang--Mills instantons were self-dual or anti-self-dual (indeed the papers \cite{BourguignonJ/LawsonH, BourguignonJ/LawsonH/SimonsJ} were a major contribution to the discussion) so it was natural to look for features exhibited by any solution. We describe in Section~\ref{section2} an observation from that period showing that, using solutions to the twistor equation, a non-self-dual instanton has at least a 4-dimensional space of solutions (of either chirality) to the Dirac equation coupled to the adjoint representation. The existence of non-self-dual solutions has since been proved \cite{SadunL/SegertJ, SibnerL/SibnerR/UhlenbeckK} and since the latter example is reasonably explicit this gives a construction.

The second case relates to the Dirac operator on a Riemann surface, which is the $\bar\partial$-operator acting on sections of a square root $K^{1/2}$ of the canonical bundle. Coupling to a vector bundle~$E$ with connection, the null space can be identified with holomorphic sections of $E\otimes K^{1/2}$. If~$E$ has degree zero the index theorem (here the Riemann--Roch theorem) gives no information about existence of solutions. When~$E$ is associated to a symplectic representation of a group $G$ then the moment map applied to a non-zero section $\psi$ of $E\otimes K^{1/2}$, when it exists, gives a Higgs field~-- a section of $\lie{g}\otimes K$. These are important objects and, following the suggestion of physicists, this issue was investigated further in~\cite{Hitchin2}. Here, in Section~\ref{section3}, we recall from that paper the construction of $\psi$ for an irreducible symplectic representation of ${\rm SU}(2)$, using points of finite order in an abelian variety.

By contrast, when $E$ arises from a representation with an orthogonal structure, and has odd rank, the mod~2 index theorem of Atiyah and Singer guarantees the existence of such a $\psi$ depending on whether the spin structure is of odd or even type. We investigate in Section~\ref{section4} the construction of such sections when~$E$ is the adjoint bundle for the group ${\rm SU}(2)$, using some classical descriptions of moduli spaces of stable bundles in genus 2.

\section{The Yang--Mills equations}\label{section2}
Let $E$ be a vector bundle on a Riemannian 4-manifold $M$. A connection $A$ on $E$ defines an exterior covariant derivative ${\rm d}_A\colon \Omega^p(M, E)\rightarrow \Omega^{p+1}(M, E)$ and ${\rm d}_A^2=F_A\in \Omega^2(M,\End E)$ is its curvature. The Yang--Mills equation for~$A$ is ${\rm d}_A\ast F_A=0$, where $\ast$ is the Hodge star operator on 2-forms. Since $\ast^2=1$ on 2-forms we have a decomposition into eigenspaces $\Lambda^2T^*=\Lambda^2_+\oplus \Lambda^2_-$ and accordingly $F_A=F_A^++ F_A^-$ so the equation is ${\rm d}_A(F_A^+- F_A^-)=0$. The Bianchi identity ${\rm d}_A(F_A^++F_A^-)=0$ implies that the Yang--Mills equation is equivalent to ${\rm d}_AF_A^+=0={\rm d}_AF_A^-$ and so if $F_A^-=0$, meaning the connection is self-dual, we have automatically a solution.

When $M$ is a spin manifold we have a spinor bundle $S=S^+\oplus S^-$ with connection induced from the Levi-Civita connection and the Clifford action $T^*\otimes S^+\rightarrow S^-$, $T^*\otimes S^-\rightarrow S^+$. The Dirac operator $\slashed D$ on $S^+$ is then defined as the composition $C^{\infty}(S^+)\stackrel{\nabla}\rightarrow C^{\infty}(T^*\otimes S^+)\rightarrow C^{\infty}(S^-)$. Using a connection on an auxiliary vector bundle $E$ to define one on $ S^+\otimes E$ we obtain the {\it coupled} Dirac operator $\slashed D_A\colon C^{\infty}(S^+\otimes E)\rightarrow
C^{\infty}(S^-\otimes E)$.

The index of this operator is $\dim \ker \slashed D_A-\dim \coker \slashed D_A=\mathrm{ch}(E)\hat A[M]$. When the connection is self-dual, the curvature $F_A=F_A^+$ acts trivially in the Weitzenb\"ock formula for $\slashed D_A^*$ leaving just one quarter of the scalar curvature and so if $M$ has positive scalar curvature, like $M=S^4$, $\ker \slashed D_A^*=0$. Then the dimension of the null space of $\slashed D_A$ is given by the index.

There is another differential operator defined on spinors, the {\it twistor operator} $D$. This is obtained from $\nabla\colon C^{\infty}(S)\rightarrow C^{\infty}(T^*\otimes S)$ by orthogonal projection onto the kernel of $T^*\otimes S\rightarrow S$. In terms of a local orthonormal basis this is
\begin{displaymath}
D\psi=\sum_{i=1}^4 \big(e_i\otimes \nabla_i\psi +\tfrac{1}{4}e_i\otimes e_i\cdot \slashed D\psi\big),
\end{displaymath}
where the dot denotes Clifford multiplication.
This is an overdetermined system but on the sphere $S^4$ it has a 4-dimensional space of solutions for $\psi\in S^+$ and similarly for $S^-$. It is in fact conformally invariant and by stereographic projection the solutions on $\R^4$ are affine linear spinors $x\cdot \psi_1+\psi_2$, where $\psi_1$, $\psi_2$ are constant spinors.

The Dirac operator is also conformally invariant acting on spinors with weight 3/2 \cite{Hitchin1} whereas the spinors for the twistor operator have weight $-1/2$. The Yang--Mills equations are conformally invariant because the Hodge star is conformally invariant in the middle dimension. In the context of conformal invariance of operators this mean that the exterior derivative on 2-forms has conformal weight~$2$. A~2-form acts on a spinor by Clifford multiplication so that a~2-form with weight 2 acting on a spinor with weight $-1/2$ gives a spinor with weight $3/2$. The following proposition is thus conformally invariant:

\begin{proposition}\samepage Let $F_A^-$ be the anti-self-dual component of a non-self-dual solution to the Yang--Mills equations on $S^4$, and $\psi\in C^{\infty}(S^-)$ a solution to the twistor equation. Then $F_A^-\cdot \psi\in \C^{\infty}( S^-\otimes \End E)$ satisfies the coupled Dirac equation.
\end{proposition}

\begin{proof} Let $e_1,\dots, e_4$ be a local orthonormal basis of 1-forms with volume form $e_1\wedge e_2\wedge e_3\wedge e_4$. Then a basis of anti-self-dual 2-forms is given by
$e_1\wedge e_2-e_3\wedge e_4$, $e_2\wedge e_3-e_1\wedge e_4$, $e_3\wedge e_1-e_2\wedge e_4$.
The Clifford algebra has relations $e_i\cdot e_j+e_j\cdot e_i=-2\delta_{ij}$ and the action of $F_A^-$ on $\psi$ is through the elements $e_1\cdot e_2-e_3\cdot e_4$ etc.\ in the algebra. Using this identification of the exterior forms with the Clifford algebra (as vector spaces) we have the following identities: if $\omega$ is an anti-self-dual 2-form and $a$ a 1-form then
\begin{displaymath}a\cdot\omega = a\wedge \omega-\ast(a\wedge \omega),\qquad \sum_{i=1}^4e_i\cdot \omega \cdot e_i=0.\end{displaymath}
Now
\begin{displaymath}\slashed D_A(F_A^-\cdot \psi)=\sum_{i=1}^4 e_i\cdot \nabla_iF_A^-\cdot \psi +e_i\cdot F_A^-\cdot \nabla_i \psi=\sum_{i=1}^4 e_i\cdot \nabla_iF_A^-\cdot \psi -\tfrac{1}{4}e_i\cdot F_A^-\cdot e_i \cdot \slashed D\psi,\end{displaymath}
using the fact that $\psi$ satisfied the twistor equation $\nabla_i\psi=-e_i\cdot \slashed D\psi /4$.
But $\sum_{i=1}^4e_i\cdot \nabla_iF_A^-={\rm d}_AF_A^--\ast {\rm d}_AF_A^-$ from the first identity and ${\rm d}_AF_A^-=0$ from the Yang--Mills equations so the first term is zero. The second term vanishes using the second identity.
\end{proof}
\begin{remark}

 In \cite{SibnerL/SibnerR/UhlenbeckK}, a non-self-dual solution was produced by a min-max argument for the trivial bundle on $S^4$. The index is zero here and so there is a canonical determinant section of the determinant line bundle which vanishes on those connections for which there exists a~solution to the coupled Dirac equation. The proposition shows that all non-self-dual Yang--Mills connections lie on this divisor. In fact since $\End E$ is a real bundle and $S^-$ is quaternionic the space of solutions is even dimensional so the determinant section is the square of a section of a Pfaffian bundle.\looseness=-1
\end{remark}

\begin{remark}
 The non-self-dual solution given in \cite{SadunL/SegertJ} has a high degree of symmetry -- it is invariant by the action of ${\rm SO}(3)$ on the unit sphere in the 5-dimensional irreducible representation on trace-free symmetric $3\times 3$ matrices. The proof of existence reduces to analysis of the boundary conditions for a system of ODEs. In principle we can write down the coupled spinor fields constructed as above but the resulting formula is not particularly rewarding.
\end{remark}
\begin{remark}
 If the connection is anti-self-dual then the proposition still yields solutions by applying the curvature $F_A=F_A^-$ to solutions of the twistor equation. Note that the index theorem gives in this case the dimension of the null space as $4k$, where $k$ is the Chern class \cite[Section~10]{AtiyahM/HitchinN/SingerI}. In particular for $k=1$ we see that all four solutions are produced this way. The connection in question is here the Levi-Civita connection on $S^-$ for a standard constant curvature metric on~$S^4$ .
\end{remark}

\section{Riemann surfaces -- the symplectic case}\label{section3}
The conformal invariance of the Dirac equation in two dimensions has a more important consequence since it can be defined in terms of the holomorphic structure and so becomes part of the algebraic geometry of a compact projective curve $C$ of genus $g$. A spin structure is defined by the holomorphic structure on a square root $K^{1/2}$ of the canonical bundle $K$ and the Dirac operator is just $\slashed D=\bar\partial\colon C^{\infty}\big(K^{1/2}\big)\rightarrow C^{\infty}\big(K^{1/2}\bar K\big)$.
In this case, the index theorem, which is the Riemann--Roch theorem, gives $\dim H^0\big(C,K^{1/2}\big)-\dim H^1\big(C, K^{1/2}\big)=0$.
When coupled to a vector bundle $E$ with connection, the $(0,1)$ component of the connection defines a $\bar\partial$-operator on $E$ and the null space of the coupled Dirac operator is $H^0\big(C,E\otimes K^{1/2}\big)$ and if $c_1(E)=0$ the index is again zero.

Suppose we are given a holomorphic principal $G$-bundle over $C$ and we take a symplectic representation to yield a vector bundle $E$. The moment map for a symplectic action of $G$ on a vector space $V$ is a quadratic map from $V$ to $\lie{g}^*$, the dual of the Lie algebra. Then, using a~bi-invariant metric, a holomorphic section of $E\otimes K^{1/2}$ gives a section of $\lie{g}\otimes K$. These are called Higgs fields and the moduli space of bundles together with Higgs fields has gained importance in various fields. The construction above of special Higgs fields was suggested by the physicist D.~Gaiotto \cite{Gaiotto} and described in~\cite{Hitchin2}.

We shall consider the irreducible symplectic representations of ${\rm SL}(2,\C)$. These are the symmetric powers $S^m\C^2$ for odd $m$ of the basic representation on $\C^2$ or equivalently the polynomials in one variable of odd degree $m$ under the action
$p(z)\mapsto (cz+d)^mp((az+b)/(cz+d)).$
 Set $m=2k-1$ and write $p(z)=a_0z^m+a_1z^{m-1}+\dots+a_m$, then the invariant symplectic form on the space $S^m\C^2$ is up to a constant
\begin{equation}
\omega=\sum_{\ell=0}^{k-1}(-1)^{\ell}\ell!(m-\ell)!{\rm d}a_i\wedge {\rm d}a_{m-i}.\label{symp}
\end{equation}
The even symmetric powers $S^{m}\C^2$ are the orthogonal irreducible representations and $m=2$ is the adjoint representation. The moment map can be viewed as $\mu\colon S^m\C^2\otimes S^m\C^2\rightarrow S^2\C^2$ given by contraction using the symplectic form.

 First consider $m=1$.
We start with a rank 2 holomorphic vector bundle $V$ with $\Lambda^2V$ trivial. Multiples of a holomorphic section $\psi$ of $V\otimes K^{1/2}$ generate a subbundle $L^*\subset V$ with quotient~$L$. Then $\psi$ is a section of the line bundle $L^*K^{1/2}$. This has degree $g-1-\deg L$. So to construct such a~$\psi$ we take an effective divisor $D$ of degree $< g-1$ and define $L$ to be $K^{1/2}(-D)$ and the vector bundle to be given by an extension class in $H^1\big(C,L^{-2}\big)=H^1(C,K(-2D))$.

The moment map in this case gives an everywhere nilpotent Higgs field $\Phi=\psi\otimes \psi$, but for $m>1$ this is not necessarily so and we can make use of the relation to Higgs bundles and adopt the spectral curve approach.
We start with a holomorphic section $q$ of $K^2$ with simple zeros and define the curve $S$ in the total space $\vert K\vert$ of $K$ (i.e., the cotangent bundle of $C$) by the equation $x^2=\pi^*q$. Here $\pi\colon \vert K\vert \rightarrow C$ is the projection and $x$ is the tautological section of $\pi^*K$ on the total space $\vert K\vert $. Then $S$ is a smooth curve of genus $4g-3$ and has a covering involution $\sigma$ defined by $x\mapsto-x$. The direct image of a line bundle $L$ on $S$ is a rank 2 bundle $V$ on $C$ and if $L\cong U\pi^*K^{1/2}$ and $\sigma^*U\cong U^*$ then $\Lambda^2V$ is trivial (see for example~\cite{BeauvilleA/NarasimhanM/RamananS}). The condition on $U$ can be interpreted as saying it lies in the Prym variety of the map $\pi\colon S\rightarrow C$, an abelian variety of dimension $4g-3-g=3g-3$.

Proposition~3 of \cite{Hitchin2} shows how to find a solution to the Dirac operator coupled to $S^mV$ from a line bundle of order $m$ on $S$ -- a point of finite order in the Prym variety. For simplicity we consider the value $m=3$. For our purposes we are not interested in the Higgs field, only its spectral curve.

\begin{proposition} Let $U$ be a non-trivial line bundle on $S$ such that $U^3$ is trivial. Then the direct image $V$ of $U\pi^*K^{1/2}$ admits a holomorphic section of $S^3V\otimes K^{1/2}$.
\end{proposition}

\begin{proof} The definition of direct image $V=\pi_*L$ is that for each open set $W\subset C$, $H^0(W, \pi_*L)=H^0\big(\pi^{-1}(W), L\big)$ so that evaluation of a local section of $L$ expresses the pull-back $\pi^*V$ as an extension{\samepage
\begin{equation}
0\rightarrow L^*\rightarrow \pi^*V\rightarrow L\rightarrow 0.
\label{V1}
\end{equation}
and this induces a filtration of $\pi^*S^3V$ into subbundles $V_0\subset V_1\subset V_2\subset V_3$ with $V_0\cong L^{-3}$.}

Consider the subbundle $V_1$. Comparing with the filtration by degree of polynomials $a_0z^3+a_1z^{2}+a_2z+a_3$ it is the subspace $a_0=a_{1}=0$, the polynomials which vanish to order $2$ at $\infty$. It is maximally isotropic with respect to the symplectic form \eqref{symp} with $m=3$.
It follows that $\pi^*S^3V$ is an extension
\begin{equation}
0\rightarrow V_1\rightarrow \pi^*S^3V\rightarrow V_1^*\rightarrow 0
\label{ext}
\end{equation}

Now $\pi^*V$ is invariant by $\sigma$ so from \eqref{V1} we have another extension
\begin{displaymath}0\rightarrow \sigma^*L^*\rightarrow \pi^*V\rightarrow \sigma^*L\rightarrow 0\end{displaymath}
and this splits \eqref{V1} outside the divisor given by $x=0$. The extension class is then supported on $x=0$, the fixed point set of $\sigma$, where $L\cong \sigma^*L$. In Dolbeault terms
we have a section of $L\sigma^*L^*$ on $x=0$, extend it holomorphically to a neighbourhood and then everywhere as a $C^{\infty}$ section $s$ supported in a neighbourhood of the fixed point set. The extension class in $H^1\big(S,L^{-2}\big)$ is then represented by $\bar\partial s/x\in \Omega^{0,1}\big(S, L^{-2}\big)$.

The Dolbeault representative for the holomorphic structure of $V$ as an extension gives a $C^{\infty}$-splitting $\pi^*V=L\oplus L^*$ and a $\bar\partial $-operator $\bar\partial +\alpha$, where $\alpha\in \Omega^{0,1}\big(S, L^{-2}\big)$. Then $\pi^*S^3V=L^{3}\oplus L\oplus L^{-1}\oplus L^{-3}$ and a holomorphic section of $\pi^*\big(S^3V\otimes K^{1/2}\big)$ can be written as $a_0z^3+a_1z^{2}+a_1z+a_0$, where
\begin{displaymath}\bar\partial a_0=0,\qquad \bar\partial a_1+3\alpha a_0=0,\qquad \bar\partial a_2+2\alpha a_1=0, \qquad \bar\partial a_3+\alpha a_2=0\end{displaymath}
and $a_0\in \Omega^0\big(S, L^3\pi^*K^{1/2}\big)$, $a_1\in \Omega^0\big(S, L\pi^*K^{1/2}\big)$, etc.

The line bundle $U^3$ is trivial. Let $u$ be a trivialization such that $\sigma^*u=u^{-1}$ and set $a_0=ux^2$. Then
$3\alpha a_0=3u(\bar\partial s) x=\bar\partial (3usx)$
so we can take $a_1=-3usx$ and extend $a_0$ to get a holomorphic section $v$ of $V_1^*\otimes \pi^*K^{1/2}$. This vanishes on $x=0$ because $a_0$ and $a_1$ do. Any two choices of $a_1$ differ by an element of $H^0\big(S, L\pi^*K^{1/2}\big)$ and if they are divisible by $x$ then this is equivalent to sections of $L\pi^*K^{-1/2}=U$. Thus if $U$ is non-trivial there is a unique extension divisible by $x$.

To extend $v$ to a section of $\pi^*\big(S^3V\otimes K^{1/2}\big)$ we observe that the vector bundle extension~\eqref{ext} is defined by a class in $H^1\big(S, \Hom(V_1^*,V_1)\big)$ also of the form $\bar\partial a/x$ and since $v$ as constructed is divisible by~$x$,~$v$ extends. In this case any two extensions differ by $v'\in H^0\big(S,V_1\otimes \pi^*K^{1/2}\big)$.

But we have
 \begin{displaymath}0\rightarrow L^{-3}=V_{0}\rightarrow V_{1}\rightarrow L^{-1}\rightarrow 0\end{displaymath}
and $L^{-3}\cong \pi^*K^{-3/2}$, so $V_0\otimes \pi^*K^{1/2}$ is a line bundle of negative degree and so has no sections. Moreover $L^{-1}\pi^*K^{1/2} \cong U^{-1}$ has degree zero and is non-trivial so this has no holomorphic sections either. It follows from the exact cohomology sequence that $H^0\big(S, V_1\otimes \pi^*K^{1/2}\big)=0$ and there is a unique extension.

 This gives us a section $v''$ of $\pi^*\big(S^3V\otimes K^{1/2}\big)$ on $S$. But our choice of trivialization $u$ implies that $\sigma$ takes $ux^2$ to $u^{-1}x^2$ and by uniqueness $\sigma^*v''=v''$. Since $\pi^*V$ is pulled back from~$C$, $v''$~descends to a section $\psi$ of $S^3V\otimes K^{1/2}$.
\end{proof}

\begin{remark}
Proposition 2 in \cite{Hitchin2} is the converse of this: a generic section of $S^3V\otimes K^{1/2}$ is obtained this way.
 \end{remark}
\begin{remark}
In the same article it is also shown that the curve $S$ in the above construction is in fact the spectral curve of the Higgs field~$\mu(\psi)$.
\end{remark}

\section{Riemann surfaces -- the orthogonal case}\label{section4}
\subsection{Parity}\label{section4.1}
If we attempt a construction like the one above for orthogonal representations of ${\rm SL}(2,\C)$ then the homomorphism for $m$ even $S^m\C^2\otimes S^m\C^2\rightarrow S^2\C^2$ is skew-symmetric. Then we need at least two solutions of the coupled Dirac operator to combine into a Higgs field. We shall not pursue this here but instead focus on the parity of the dimension of the null space. The integer index theorem via Riemann--Roch gives us no information but the mod~2 index theorem~\cite{AtiyahM/SingerI} does:
\begin{proposition} \label{mod2} Let $E$ be a vector bundle over $C$ with a non-degenerate symmetric bilinear form. If $E$ has odd rank and Stiefel--Whitney class $w_2(E)=0$ then
$\dim H^0\big(C, E\otimes K^{1/2}\big) = \dim H^0\big(C,K^{1/2}\big)$ {\rm mod}~$2$. If $w_2(E)\ne 0$ then the dimension is $1+\dim H^0\big(C,K^{1/2}\big)$ {\rm mod}~$2$.
\end{proposition}
Once one has established that $\dim H^0\big(C, E\otimes K^{1/2}\big)$ {\rm mod 2} is a $C^{\infty}$ deformation invariant then this can be proved by deforming $E$ continuously to be a trivial bundle $\C^{2m+1}$ in the $w_2(E)=0$ case. Then
$\dim H^0\big(C, E\otimes K^{1/2}\big)$ {\rm mod 2} = $(2m+1)\dim H^0\big(C,K^{1/2}\big)$ {\rm mod}~2 = $\dim H^0\big(C,K^{1/2}\big)$ {\rm mod} 2
and similarly for the other case. This is described in~\cite{Oxbury}.

In the case of $K^{1/2}$ alone (the basic Dirac operator) these square roots are classically known as theta characteristics and of the $2^{2g}$ choices of square root, or spin structure, there are $2^{g-1}(2^g-1)$ odd ones and $2^{g-1}(2^g+1)$ even ones. When the parity is odd then we know that there is a~non-zero solution to the Dirac equation.

We shall restrict ourselves to the adjoint representation of ${\rm SL}(2,\C)$, or the rank $3$ vector bundle $\End_0V$ of trace free endomorphisms of a rank $2$ bundle $V$ with symmetric form $\tr AB$. Here $w_2(\End_0V)=\deg V$ mod~2.

\begin{remark}
 It is possible to give a spectral curve description of bundles $V$ with sections $\psi$ of $\End_0V\otimes K^{1/2}$ by taking a covering of $C$ branched over the zeros of a holomorphic one-form as in \cite{Oxbury}, rather than a quadratic differential as in the previous section, but spectral curve methods have the disadvantage of not separating the roles of bundle and section. In the case of Higgs bundles, fixing a stable bundle and fixing a spectral curve give two different Lagrangian submanifolds whose intersection is in general difficult to define, but is of great interest.
 \end{remark}
 In what follows we shall consider the case where $C$ has genus $2$ and try and construct the solutions which are guaranteed by parity.

\subsection{Genus 2 even degree}\label{section4.2}
Let $C$ be a genus $2$ curve expressed as a double covering of $\PP^1$ branched over six points $x_1,\dots, x_6$. We have $2^{g-1}\big(2^g-1\big)=6$ odd spin structures. In higher genus the odd dimension of $H^0\big(C,K^{1/2}\big)$ can vary with the modulus of the curve but for genus~$2$ it must be one. So the line bundles~$K^{1/2}$ have degree $g-1=1$ with a single section $s$ which vanishes at one point and these are the preimages $\tilde x_i$ of the branch points. These then define solutions of the basic Dirac equation on~$C$.

The rank 2 stable bundles $V$ with $\Lambda^2V$ trivial have a moduli space which is three-dimensional and (if we include S-equivalence classes of semistable bundles) can be identified with~$\PP^3$~\cite{NarasimhanM/RamananS}. The tangent space at $[V]$ of the moduli space is naturally isomorphic to $H^1(C, \End_0V)$ and the cotangent space is $H^0(C, \End_0V\otimes K)$ by Serre duality. Since the rank of $\End_0V$ is odd Proposition~\ref{mod2} implies that we have a section $\psi$ of $\End_0V\otimes K^{1/2}$ if $K^{1/2}$ is odd. Then $\Phi=s\psi$ is a section of $\End_0V\otimes K$ which vanishes at~$\tilde x_i$. For generic bundles $E$ there is just a one-dimensional space of sections and this in fact holds in this low genus case.
 We thus have a~distinguished one-dimensional subspace in each cotangent space of~$\PP^3$, at least over the open set of stable bundles.

To find this consider $\alpha \in H^1\big(C,K^{-1}\big)$ and the homomorphism
\begin{displaymath}\alpha\colon \ H^0(C, \End_0 V\otimes K)\rightarrow H^1(C,\End_0 V).\end{displaymath}
The dual space of $H^1\big(C,K^{-1}\big)$ is $H^0\big(C,K^2\big)$ and evaluation of a section of~$K^2$ at~$\tilde x_i$ defines up to a multiple a class~$\alpha_i$ whose kernel is the required one-dimensional space. An alternative description is to take the Serre duality pairing $\big\langle \alpha, \tr \Phi^2\big\rangle$ to define a conic in each projective plane $\PP\big(H^0(C,\End_0 V\otimes K)\big)$ and for $\alpha_i$ this is a pair of lines whose intersection defines the one-dimensional space. Evaluation of this quadratic function was achieved in a very explicit manner in~\cite{Gawedzki} following basic work in~\cite{GeemenB/PreviatoE}.

A section of $K^2$ may be written as $\big(a_0+a_1x+a_2x^2\big){\rm d}x^2/y^2$, where the equation of $C$ is $y^2=(x-x_1)\cdots (x-x_6)$. The cotangent bundle of $\PP^3=\PP(U)$ is written as the (symplectic) quotient $\{(p,q)\in U^*\times U\colon \langle p, q \rangle =0,\, q\ne 0\}/\C^*$, where $\lambda\in \C^*$ acts as $\big(\lambda^{-1} p, \lambda q\big)$. So $q_i$ are homogeneous coordinates in $\PP^3=\PP(U)$ and the $p_i$ linear in the fibres of $T^*\PP^3$.
Then up to a~factor the quadratic map $\tr \Phi^2$ is given in~\cite{Gawedzki} by
\begin{displaymath}h=\sum_{i\ne j}\frac{r_{ij}(q,p)}{(x-x_i)(x-x_j)}{\rm d}x^2=\sum_{i\ne j}\prod_{k\ne i,j}(x-x_k){r_{ij}(q,p)}\frac{{\rm d}x^2}{y^2},\end{displaymath}
where
$r_{12}(q,p)= (q_1p_1 + q_2p_2-q_3p_3-q_4p_4)^2$ and
\begin{alignat*}{3}
& r_{13}(q,p)= (q_1p_4-q_2p_3-q_3p_2 + q_4p_1)^2,\qquad &&
r_{14}(q, p) = -(q_1p_4 + q_2p_3-q_3p_2-q_4p_1)^2,&\\
& r_{15}(q, p) = -(q_1p_3-q_2p_4-q_3p_1 + q_4p_2)^2,\qquad &&
r_{16}(q,p)= (q_1p_3 + q_2p_4 + q_3p_1 + q_4p_2)^2,& \\
& r_{23}(q, p) = -(q_1p_4-q_2p_3 + q_3p_2-q_4p_1)^2,\qquad &&
r_{24}(q,p)= (q_1p_4 + q_2p_3 + q_3p_2 + q_4p_1)^2,& \\
& r_{25}(q,p)=(q_1p_3-q_2p_4 + q_3p_1-q_4p_2)^2,\qquad &&
r_{26}(q, p) = -(q_1p_3 + q_2p_4-q_3p_1-q_4p_2)^2, &\\
& r_{34}(q, p) = (q_1p_1-q_2p_2 + q_3p_3-q_4p_4)^2,\qquad &&
r_{35}(q, p) = (q_1p_2 + q_2p_1 + q_3p_4 + q_4p_3)^2, &\\
& r_{36}(q, p) = -(q_1p_2-q_2p_1 -q_3p_4 + q_4p_3)^2, \qquad &&
r_{45}(q, p) = -(q_1p_2-q_2p_1 + q_3p_4-q_4p_3)^2, &\\
& r_{46}(q, p) = (q_1p_2 + q_2p_1-q_3p_4-q_4p_3)^2,\qquad &&
r_{56}(q, p) = (q_1p_1-q_2p_2-q_3p_3 + q_4p_4)^2.&
\end{alignat*}
Evaluating at $x=x_1$ gives a linear combination of the terms $r_{1i}$. Each of these is a quadratic form $\ell_i\otimes \ell_i$, where $\ell_i(p_1,p_2,p_3,p_4)$ is a linear form. But on examination one sees that $\ell_i(q_2, {-}q_1, q_4,\allowbreak {-}q_3) =0$ for each $i$.
 Hence

\begin{proposition} At the point $[q_1,q_2,q_3,q_4]\in \PP^3$ representing the vector bundle $V$, the cotangent vectors $(p_1,p_2,p_3,p_4)=(q_2, -q_1, q_4, -q_3)$
define the one-dimensional space of Higgs fields which vanish at $\tilde x_1$ or equivalently the holomorphic sections of $\End_0V\otimes K^{1/2}$ for each $V$.
\end{proposition}
\begin{remark} The evident symmetry in the formulae above arises from the action $V\mapsto V\otimes U$ of the $2^{2g}=16$ line bundles $U$ with $U^2$ trivial.
\end{remark}
\subsection{Genus 2 odd degree}\label{section4.3}
Suppose now the rank 2 bundle has odd degree. From Proposition \ref{mod2}, $H^0\big(C,\End_0V\otimes K^{1/2}\big)$ has a section if $K^{1/2}$ is one of the $10 =16 - 6$ even spin structures. The latter have no non-zero holomorphic sections.

We use a model of the curve $C$ using $K\otimes K^{1/2}=K^{3/2}$. By Riemann--Roch $K^{3/2}$ has a~two-dimensional space of holomorphic sections as does $K$ and hence together they give a map $C\rightarrow \PP^1\times \PP^1$ so that $C$ is a divisor of $\OO(2,3)$. The line bundle $\OO(1,1)$ embeds $\PP^1\times \PP^1\subset \PP^3$ and so a plane in~$\PP^3$ intersects $C$ in a divisor of~$K^{5/2}$. By Riemann--Roch $\dim H^0\big(C,K^{5/2}\big)=4$ so all such divisors have this form.

Now we need a description of vector bundles of odd degree. Note that $\End_0V$ does not change if we replace $V$ by $V\otimes L$ for a line bundle~$L$, which means we can make use of Atiyah's description of projective bundles in~\cite{Atiyah}.

The construction takes a representative vector bundle as an extension $\OO\rightarrow V\rightarrow L$, where~$L$ has degree~$1$. The extension class lies in a two-dimensional space $H^1(C,L^*)$ and up to a scalar is determined by its annihilator under Serre duality, a section $s$ of $LK$. This vanishes at three points $p$, $q$, $r$ in general distinct. There are four degree zero subbundles and the same projective bundle $\PP(V)$ is defined by three other triples $(p, \sigma(q),\sigma(r)),(\sigma(p),q,\sigma(r)), (\sigma(p),\sigma(q),r)$, where $\sigma\colon C\rightarrow C$ is the hyperelliptic involution.

Given $\psi\in H^0\big(C,\End_0V\otimes K^{1/2}\big)$, its action on the trivial subbundle followed by the projection onto the quotient $L$ gives a holomorphic section $c$ of $LK^{1/2}$. As divisor classes $LK\sim p+q+r$ so $LK^{1/2}\cong K^{-1/2}(p+q+r)\cong K^{5/2}(-\sigma(p)-\sigma(q)-\sigma(r))$ since $x+\sigma(x)$ is the divisor of a~section of~$K$. Thus $c$ is defined by a~plane in $\PP^3$ which passes through the three points $\sigma(p)$, $\sigma(q)$, $\sigma(r)$.
\begin{proposition} Let $V$ be a vector bundle defined by three points $p,q,r\in C\subset \PP^3$. If these points are not collinear, then the plane though $\sigma(p)$, $\sigma(q)$, $\sigma(r)$ uniquely determines a section $\psi \in H^0\big(C,\End_0V\otimes K^{1/2}\big)$.
\end{proposition}
\begin{proof}
We use the Dolbeault approach for the extension $\OO \rightarrow E\rightarrow L$ by a $C^{\infty}$ splitting and a~form $\alpha\in \Omega^{0,1}(C, L^*)$ defining the extension. Then
\begin{displaymath}\psi=
\begin{pmatrix} a & \hphantom{-}b \cr
c & -a
\end{pmatrix}
\end{displaymath}
is holomorphic if
$\bar\partial c=0$, $\bar\partial a+\alpha c=0$, $\bar\partial b - 2\alpha a=0$,
where $c$ is a holomorphic section of~$LK^{1/2}$, $a$ is a $C^{\infty}$ section of $K^{1/2}$ and $b$ is a section of~$L^*K^{1/2}$.

As above, the plane defines $c$, a section of $K^{5/2}(-\sigma(p)-\sigma(q)-\sigma(r))$ and since $\sigma$ induces a~projective transformation of~$\PP^3$, $\sigma(p)$, $\sigma(q)$, $\sigma(r)$ are not collinear by assumption. Thus there is a unique plane passing though them, or equivalently $\dim H^0\big(C,LK^{1/2}\big)=1$, generated by~$c$.

The term $\alpha c$ represents a class in $H^1\big(C,K^{1/2}\big)$ which by Riemann--Roch has the same dimension as $H^0\big(C,K^{1/2}\big)$. But this is zero since $K^{1/2}$ is an even
spin structure. Then we can solve $\bar\partial a+\alpha c=0$ for $a$. Any two choices differ by a holomorphic section of $K^{1/2}$ so $a$ is unique.

To find $b$ we need $\alpha a$ to define a trivial class in $H^1\big(C, L^*K^{1/2}\big)$. This is dual to the space $H^0\big(C,LK^{1/2}\big)$ which consists of multiples of~$c$.
 So $[\alpha a]=0$ if $\alpha a$ pairs to zero with~$c$.
 But since $\bar\partial a+\alpha c=0$
\[
\int_C (\alpha a) c= \int_C a (\alpha c)=-\int_Ca\big(\bar\partial a\big) =-\int_C\bar\partial \big(a^2/2\big)=0.
\]
Then we can solve $\bar\partial b - 2\alpha a=0$ for $b$. Since $L^*K^{1/2}$ is of degree zero and non-trivial, $b$ is unique.
\end{proof}
\begin{remark}
Note that if the points are collinear, $L K^{1/2}$ has degree $2$ and a two-dimensional space of sections and so must be isomorphic to $K$, or equivalently $L\cong K^{1/2}$. There is then an obvious section $\psi$ given by $a=c=0$, $b=1$.
\end{remark}

\pdfbookmark[1]{References}{ref}
\LastPageEnding

\end{document}